\documentclass[a4paper,11pt]{article}
\usepackage[cp1250]{inputenc}
\usepackage{amssymb}
\usepackage[T1]{fontenc}
\usepackage{authblk}
\usepackage{amsmath}
\usepackage{amsthm}
\usepackage{accents}
\usepackage{enumerate}
\usepackage{bbm}
\usepackage[left=3cm,right=3cm,top=3cm,bottom=3cm]{geometry}
\usepackage{hyperref}
\hypersetup{colorlinks=true, linkcolor=blue, citecolor=red,
filecolor=black, urlcolor=black }
\newcommand{\R}{\mathbb{R}}
\newcommand{\N}{\mathbb{N}}

\newcommand{\F}{\mathcal{F}}
\newcommand{\Fb}{\mathbb{F}}
\newcommand{\T}{\mathcal{T}}
\newcommand{\e}{\varepsilon}
\newcommand{\V}{\mathcal{V}}
\newcommand{\Vp}{{^p\mathcal{V}}}
\newcommand{\Vo}{\mathcal{V}_0}
\newcommand{\Vc}{\mathcal{V}^1}
\newcommand{\M}{\mathcal{M}}

\newcommand{\ind}[1]{\mathbbm{1}_{#1}}
\renewcommand{\Pr}{P}

\newcommand{\E}{E}
\newcommand{\rs}{\mbox{\rm RBSDE}}
\newcommand{\lrs}{\underline{\mbox{\rm RBSDE}}}
\newcommand{\urs}{\overline{\mbox{\rm RBSDE}}}
\DeclareMathAccent{\wtilde}{\mathord}{largesymbols}{"65}
\newcommand{\wideutilde}[1]{\underaccent{\wtilde}{#1}}
\DeclareMathOperator*{\esssup}{ess\,sup}
\DeclareMathOperator*{\essinf}{ess\,inf}

\newcommand{\cadlag}{c\`adl\`ag }

\theoremstyle{definition}
\newtheorem{example}{Example}[section]
\newtheorem{definition}{Definition}[section]
\newtheorem{remark}{Remark}[section]
\theoremstyle{plain}
\newtheorem{theorem}{Theorem}[section]
\newtheorem{lemma}{Lemma}[section]
\newtheorem{corollary}{Corollary}[section]

\numberwithin{equation}{section}

\title{Reflected BSDEs with general filtration and two separated
barriers}

\author{Mateusz Topolewski}
\affil{Faculty of mathematics and Computer Science, Nicolaus
Copernicus University, Chopina 12/18, 87--100 Toruń, Poland.\\ {E-mail: \tt woland@mat.umk.pl}}

\begin{document}
\date{}
\maketitle
\begin{abstract}
We consider reflected backward stochastic differential equations,
with two barriers,  defined on probability spaces equipped with
filtration satisfying only the usual assumptions of right
continuity and completeness. As for barriers we assume that there
are \cadlag processes of class D that are completely separated. We
prove the existence and uniqueness of solutions for integrable
final condition and integrable monotone generator. An application
to zero-sum Dynkin game is given.
\end{abstract}

{\noindent\it MSC:} Primary 60H10; Secondary 60H30, 60H99.

\vspace{5pt}

{\noindent\it Keywords:} Reflected BSDE; general filtration;
separated barriers; $L^1$ data

\section{Introduction and notation}

In this paper we study the problem of existence and uniqueness of
solutions of backward stochastic differential equations with two
reflecting c\'adl\'ag barriers $L,U$. The main new feature is that
we deal with equations on probability spaces with general
filtration $\mathbb{F}=\{\F_t;t\in[0,T]\}$ satisfying only the
usual conditions of right-continuity and completeness, and we do
not assume that the barriers satisfy  so-called Mokobodzki
condition.  Instead, we assume that the lower barrier $L$ and the
upper barrier $U$ are completely separated in the sense that
$L_t<U_t$ and $L_{t-}<U_{t-}$ for $t\in [0,T]$. Moreover, we
consider equations with $L^p$ data, where $p\in[1,2]$. Our
motivation for considering such general setting comes from PDEs
theory (equations involving nonlocal operators, see \cite{kliroz,
kliroz2013}) and from the theory of optimal stopping (Dynkin
games, see \cite{klimsiak2014,Lepeltier2007,Stet1982,Zab1984}).

Let $T>0$. Suppose we are given an $\F_T$--measurable random
variable $\xi$, a~progressively measurable function
$f:\Omega\times[0,T]\times\R\to\R$ and two  adapted \cadlag
processes $L,U$  such that $L_t\leq U_t$, $t\in[0,T]$. Roughly
speaking, by a solution of the reflected BSDE with terminal
condition $\xi$, generator $f$ and barriers $L,U$ we mean a
quadruple $(Y,K,A,M)$ of \cadlag adapted processes such that $Y$
is  of Doob's class D, $K,A$ are increasing processes such that
$K_0=A_0=0$, $M$ is a~local martingale with $M_0=0$, and a.s.  we
have
\begin{equation}
\label{eq1.1} \left\{
\begin{array}{l}
Y_t=\xi+\int_t^Tf(s,Y_s)\,ds +\int_t^T dK_s
- \int_t^T dA_s -\int_t^TdM_s, \quad t\in[0,T]\smallskip\\
L_t\leq Y_t\leq U_t,\quad t\in[0,T],
\smallskip\\
\int_0^T (Y_{t-}-L_{t-})\,dK_t=\int_0^T(U_{t-}-Y_{t-})\,dA_t=0.
\end{array}
\right.
\end{equation}

In most papers devoted to reflected BSDEs with two barriers it is
assumed that $L,U$ satisfy one of the following conditions:
\begin{itemize}
\item[(a)]between $L$ and $U$ one can find a process $X$ such that
$X$ is a difference of  nonnegative c\`adl\`ag supermartinagles
(so-called Mokobodzki condition), or
\item[(b)]$L_t<U_t$ and
$L_{t-}<U_{t-}$ for $t\in[0,T]$ (i.e. the  barriers are completely
separated).
\end{itemize}

Problem (\ref{eq1.1}) under  assumption (a) is studied thoroughly
in Klimsiak \cite{klimsiak2014}. Among other things,  in
\cite{klimsiak2014} it is proved that if $f$ is continuous and
monotone with respect to $y$ and satisfies mild integrability
conditions (see hypotheses (H1)--(H4) in Section \ref{sec2}), then
there exists a unique solution of (\ref{eq1.1}).

A drawback to  assumption (a), and one of the main reason why more
explicit condition (b) is considered, is that (a)  may be
sometimes difficult to check. Unfortunately, equations with
barriers satisfying (b) are more difficult to deal with. At
present,  all the existing results on equations with barriers
satisfying (b) concern the case where the underlying filtration is
Brownian (see Hamad\`ene and Hassani \cite{hamhas}, Hamad\`ene,
Hassani and Ouknine \cite{hamhasou}) or is generated by a Brownian
motion and an independent Poisson random measure (see Hamad\`ene
and Wang \cite{hamadenewang}). Moreover, in
\cite{hamhas,hamhasou,hamadenewang}) it is assumed that $f$ is
Lipschitz continuous and the data (including barriers) are
$L^2$-integrable. Recently, in \cite{Hass2016}, in the case of Brownian
filtration, an existence and uniqueness result was proved for
equations with separated continuous barriers, $L^1$ data and
Lipschitz continuous generator.

Our main theorem says that under the assumptions on $\xi,f$ from
\cite{klimsiak2014} and \cadlag barriers $L,U$ satisfying (b) and
such that $L^+,U^-$ are of class D there exists\break a~unique solution
of (\ref{eq1.1}). Thus we extend the results from
\cite{klimsiak2014} to barriers satisfying (b) and at the same
time we generalize the results of
\cite{hamhas,hamhasou,hamadenewang,Hass2016} to equations with general
filtration and less regular data. 
It is worth pointing out that as a
simple corollary to our existence result (it suffices to consider
the generator $f\equiv0$) one gets  the following result from the
general theory of stochastic processes: if two \cadlag processes
$L,U$ are completely separated and $L^+,U^-$ are of class D, then
there exists a semimartingale of class D between $L$ and $U$.

The main idea of the proof of our main result is to reduce the
problem with  completely separated barriers to the problem with
barriers satisfying the Mokobodzki condition, and then apply the
results of \cite{klimsiak2014}. Such a reduction is possible
locally (we use here some modification of a construction from
\cite{FalkRozprawaEN}) and enables us to obtain solutions of
(\ref{eq1.1}) on stochastic intervals of the form $[0,\tau_n]$,
where $\{\tau_n\}$ is some stationary sequence of stopping times.
These local solutions can be put together to get the solution of
(\ref{eq1.1}) on $[0,T]$. The last step involves some
technicalities, but in general our proof is short and rather
simple. In our opinion, it is much simpler than the proof for
equations with the underlying  Brownian-Poison filtration and
$L^2$ data  given in \cite{hamadenewang}.

The paper is organized as follows. In Section \ref{sec2} we review
some results from \cite{klimsiak2014} concerning reflected BSDEs
with one barrier. The proof of the main result is given in Section
\ref{sec3}.  Finally, in  Section \ref{sec4} we give an
application of  results of Section \ref{sec3} to zero-sum Dynkin
game with payoff function determined by $\xi,f$ and $L,U$.
\medskip\\
{\bf Notation.} Let  $T>0$, and let
$(\Omega,\F,\Fb=\{\F_t\}_{t\in[0,T]},\Pr)$  be a filtered
probability space with filtration satisfying  the usual
assumptions of completeness and right continuity. By $\T$  we
denote the set of all $\Fb$--stopping times such that $\tau\leq
T$, and by $\T_t$, $t\in[0,T]$, the set of $\tau\in\T$ such that
$\Pr(\tau\geq t)=1$.

By $\V$ we denote the set of all $\Fb$-progressively measurable
processes of finite variation, and by $\Vc$ the subset of $\V$
consisting of all processes $V$ such that $\E|V|_T<\infty$, where
$|V|_T$ stands for the variation of $V$ on $[0,T]$. $\Vo$ is  the
subset of $\V$ consisting of all  processes $V$ such that $V_0=0$,
$\Vo^+$ (resp. $\Vp_0^+$) is the subset of $\Vo$ of all increasing
processes (resp. predictable increasing processes). $\M$ (resp.
$\M_{loc}$) denotes the set of all $\Fb$-martingales (resp. local
martingales). By $L^1(\Fb)$ we denote the space of all
$\Fb$-progressively measurable processes $X$ such that $\E\int_0^T
|X_t|dt<\infty$, and by $L^{1}(\F_T)$ the  space of all
$\F_T$-measurable random variables $\xi$ such that
$\E|\xi|<\infty$.

For a stochastic process $X$ we set $X^+=X\vee 0$, $X^-=-(X\wedge
0)$ and $X_{t-}=\lim_{s\nearrow t}X_s$ with the convention that
$X_{0-}=X_0.$  We also adopt the convention that
$\int_a^b=\int_{(a,b]}$.

\section{BSDEs with one reflecting barrier}
\label{sec2}

In what follows $\xi$ is an $\F_T$-measurable random variable,
and $L,U$ are $\Fb$-progressively measurable \cadlag processes,
$V\in\Vo$ and $f:\Omega\times[0,T]\times\R\to\R$ is a measurable
function such that $f(\cdot,y)$ is an $\Fb$-progressively
measurable process for every $y\in\R$ (for brevity, in our
notation we omit the dependence of $f$ on $\omega\in\Omega)$.

We will need the following assumptions on $\xi$ and $f$:
\begin{itemize}
\item[(H1)] There exists a constant $\mu\in\R$ such that for almost
every $t\in[0,T]$ and all $y,y'\in\R$,
\[
(f(t,y)-f(t,y'))(y-y')\leq \mu|y-y'|^2,
\]
\item[(H2)] $[0,T]\ni t\mapsto f(t,y)\in L^1(0,T)$ for every $y\in\R$,

\item[(H3)] the function $\R\ni y\mapsto f(t,y)$ is continuous  for almost
every $t\in[0,T]$,
\item[(H4)] $\xi\in L^1(\F_T)$, $V\in\Vo\cap\V^1$,
$f(\cdot,0)\in L^1(\F)$.
\end{itemize}

Recall that a stochastic process $X$ on $[0,T]$ is said to be of
class D  if $\{X_\tau:\tau\in\T\}$ is a uniformly integrable
family of random variables.

\begin{definition} We say that a triple $(Y,K,M)$ of \cadlag processes  is
a solution of the reflected BSDE with terminal condition $\xi$,
generator $f+dV$ and lower barrier $L$ ($\lrs(\xi,f+dV,L)$ for
short) if
\begin{itemize}
\item[(a)] $Y$ is a process of class D, $K\in\Vp_0^+$, $M\in\M_{loc}$
with $M_0=0$,
\item[(b)] $L_t\leq Y_t$, $t\in[0,T]$, $\Pr$-a.s.,
\item[(c)] $\int_0^T(Y_{t-}-L_{t-})\,dK_t=0$,
\item[(d)] $Y_t=\xi + \int_t^T f(s,Y_s)\,ds + \int_t^TdV_s + \int_t^T dK_s
- \int_t^T dM_s$, $t\in[0,T]$, $P$-a.s.
\end{itemize}
\end{definition}

\begin{definition}
\label{def2.2} We say that a triple $(Y,K,M)$ of \cadlag processes
is a solution of the reflected BSDE with terminal condition $\xi$,
generator $f+dV$ and upper barrier $U$ ($\urs(\xi,f+dV,U)$ for
short) if
\begin{itemize}
\item[(a)] $Y$ is a process of class D, $A\in\Vp_0^+$, $M\in\M_{loc}$ with $M_0=0$,
\item[(b)] $Y_t\leq U_t$, $t\in[0,T]$, $\Pr$-a.s.,
\item[(c)] $\int_0^T(U_{t-}-Y_{t-})\,dA_t=0$,
\item[(d)] $Y_t=\xi + \int_t^T f(s,Y_s)\,ds + \int_t^TdV_s
- \int_t^T dA_s - \int_t^T dM_s$, $t\in[0,T]$, $P$-a.s.
\end{itemize}
\end{definition}

Our motivations for considering reflected equations involving a
finite variation process $V$  comes from the theory of partial
differential equations with measure data. In these applications
$V$ is an additive functional of a Markov process in the Revuz
correspondence with some smooth measure {see
\cite{kliroz,kliroz2015,kliroz2013}

In the theorem below we recall some results on reflecting BSDEs
with one barrier proved in \cite{klimsiak2014}. They will play
important role in the proof of our main result in Section
\ref{sec3}.

\begin{theorem}\label{tw2.1}
Assume that $L^+,U^-$ are of class \mbox{\rm D} and
\mbox{\rm(H1)--(H4)} are satisfied.
\begin{itemize}
\item[\emph{(i)}] There exists a unique solution
$(\wideutilde{Y},\wideutilde{K},\wideutilde{M})$ of
$\lrs(\xi,f+dV,L)$.\break \mbox{Moreover,} if
$(\wideutilde{Y}^n,\wideutilde{M}^n)$, $n\in\N$, are solutions of
BSDEs of the form
\[
\wideutilde{Y}^n_t=\xi +\int_t^Tf(s,\wideutilde{Y}^n_s)\,ds +
\int_t^T dV_s + \int_t^T n(L_s-\wideutilde{Y}^n_s)^+\,ds -
\int_t^Td\wideutilde{M}^n_s,
\]
then $\wideutilde{Y}^n_t\nearrow \wideutilde{Y}_t$, $t\in[0,T]$
$\Pr$-a.s.
\item[\emph{(ii)}] There exists a unique
solution $(\widetilde{Y},\widetilde{A},\widetilde{M})$ of
$\urs(\xi,f+dV,U)$. Moreover, if
$(\widetilde{Y}^n,\widetilde{M}^n)$, $n\in\N$, are solutions of
BSDEs of the form
\[
\widetilde{Y}^n_t=\xi +\int_t^Tf(s,\widetilde{Y}^n_s)\,ds + \int_t^T
dV_s - \int_t^T n(\widetilde{Y}^n_s-U_s)^+\,ds -
\int_t^Td\widetilde{M}^n_s,
\]
then $\widetilde{Y}^n_t\nearrow \widetilde{Y}_t$, $t\in[0,T]$
$\Pr$-a.s.
\end{itemize}
\end{theorem}
\begin{proof}
Part (i) is proved in \cite[Theorem 4.1]{klimsiak2014} under the
assumption that $L$ is of class D. The following argument shows
that in fact it suffices to assume that $L^+$ is of class D. Let
$(Y^0,M^0)$ be a solution of the BSDE
\begin{equation}
\label{eq2.2} Y^0_t=\xi+\int_t^T f(s,Y^0_s)\,ds + \int_t^TdV_s -
\int_t^T dM^0_s,\quad t\in[0,T],
\end{equation}
and let $L^{\varepsilon}=L\vee(Y^0-\varepsilon)$ for some
$\varepsilon>0$. If $L^+$ is of class D then  $L^{\varepsilon}$ is
of class D, because $Y^0$ is of class D. Therefore by
\cite[Theorem 4.1]{klimsiak2014} there exists a solution
$(Y^{\varepsilon},K^{\varepsilon},M^{\varepsilon})$ of
$\lrs(\xi,f+dV,L^\e)$ such that $K^{\varepsilon}\in\Vo^+$. In
particular,
\begin{equation}
\label{eq2.3} Y^{\varepsilon}_t=\xi+\int_t^T
f(s,Y^{\varepsilon}_s)\,ds + \int_t^TdV_s
+\int^T_tdK^{\varepsilon}_s - \int_t^T dM^{\varepsilon}_s,\quad
t\in[0,T]
\end{equation}
and
\begin{equation}
\label{eq2.4} Y^{\varepsilon}\ge L^{\varepsilon}\ge L.
\end{equation}
By  (\ref{eq2.2}), (\ref{eq2.3}) and \cite[Proposition
2.1]{kliroz}, $Y^{\varepsilon}\ge Y^0$. Hence we have
$\ind{\{Y^{\varepsilon}_{t-}>L_{t-}\}}
=\ind{\{Y^{\varepsilon}_{t-}>L^\e_{t-}\}}$ for $t\in[0,T]$, and
consequently
\begin{align}
\label{eq2.5}
\int_0^T(Y^{\varepsilon}_{t-}-L_{t-})\,dK^{\varepsilon}_t&=
\int_0^T(Y^{\varepsilon}_{t-}-L_{t-})
\ind{\{Y^{\varepsilon}_{t-}>L_{t-}\}}(t)\,dK^{\varepsilon}_t\nonumber\\
&=\int_0^T(Y^{\varepsilon}_{t-}-L_{t-})
\ind{\{Y^{\varepsilon}_{t-}>L^\e_{t-}\}}(t)\,dK^{\varepsilon}_t
=0,
\end{align}
the last equality being a consequence of the fact that
$\int_0^T\ind{\{Y^{\varepsilon}_{t-}>L^\e_{t-}\}}(t)
\,dK^{\varepsilon}_t=0$. By (\ref{eq2.3})--(\ref{eq2.5}) the
triple $(\wideutilde{Y},\wideutilde{K},\wideutilde{M})
=(Y^{\varepsilon},K^{\varepsilon},M^{\varepsilon})$ is a solution
of the equation $\lrs(\xi,f+dV,L)$. Uniqueness follows from \cite[Corollary
2.2.]{klimsiak2014}. This proves the first part of (i). Observe
now that the first component of the solution of $\lrs(\xi,0,L)$ is
a supermartingale of class D majorizing $L$. Therefore to prove
that $\wideutilde{Y}^n_t\nearrow \wideutilde{Y}_t$, $t\in[0,T]$,
it suffices to repeat step by step the proof of \cite[Theorem
4.1]{klimsiak2014}. Since the proof of (ii) is analogous to that
of (i), we omit it.
\end{proof}

\section{BSDEs with two reflecting barriers}
\label{sec3}

In this section $\xi,f,V$ and $U,L$ are as in Section \ref{sec2}.
We also assume that $L_t\leq U_t$ for $t\in[0,T]$, $\Pr$-a.s.

\begin{definition}
We say that a quadruple $(Y,K,A,M)$ of \cadlag processes is a
solution  of the reflected BSDE with terminal condition $\xi$,
generator $f+dV$, lower barrier $L$ and upper barrier $U$
($\rs(\xi,f+dV,L,U)$ for short) if
\begin{itemize}
\item[(LU1)] $Y$ is a process of class D, $A,K\in\Vp_0^+$,
$M\in\M_{loc}$ with $M_0=0$,
\item[(LU2)] $L_t\leq Y_t\leq U_t$, $t\in[0,T]$, $\Pr$-a.s.,
\item[(LU3)] $\int_0^T(Y_{t-}-L_{t-})\,dK_t=\int_0^T
(U_{t-}-Y_{t-})\,dA_t=0$,
\item[(LU4)]
$Y_t=\xi+ \int_t^T f(s,Y_s)\,ds + \int_t^T dV_s +\int_t^T d(K_s
-A_s)- \int_t^T dM_s$, $t\in[0,T]$, $P$-a.s.
\end{itemize}
\end{definition}

We will need the following conditions for the barriers $L,U$:
\begin{itemize}
\item[(B1)] $L_t<U_t$ and $L_{t-}<U_{t-}$ for $t\in[0,T]$.
\item[(B2)] $L^+,U^-$ are  processes of class D.
\end{itemize}

A sequence $\{\tau_n\}\subset\T$ is called of stationary type, if
\[
\Pr\Big(\liminf_{n\to\infty}\{\tau_n=T\}\Big)=1.
\]
The following lemma is an extension of  \cite[Remark
3.4]{FalkRozprawaEN}.

\begin{lemma}\label{lm3.1}
Assume that $L,U$ are of class \mbox{\rm D} and satisfy
\mbox{\rm(B1)}. Then there exists a process $H\in\V$ such that
$L_t\leq H_t\leq U_t$, $t\in[0,T]$, $\Pr$-a.s. Moreover, there
exists a sequence $\{\tau_n\}\subset\T$ of stationary type  such
that $\E{|H|_{\tau_n}}<\infty$ for every $n\in\N$.
\end{lemma}
\begin{proof}
Let $\tau_0=0$, and for $n\in\N$ set
\[
\tau_n=\inf\Big\{t>\tau_{n-1}
:\frac{L_{\tau_{n-1}}+U_{\tau_{n-1}}}{2}>U_t \quad\text{or}\quad
\frac{L_{\tau_{n-1}} +U_{\tau_{n-1}}}{2}<L_t \Big\}\wedge T.
\]
Obviously $\{\tau_n\}$ is nondecreasing. We shall show that  it is
increasing up to $T$. To see this, we first observe that
\begin{equation}
\label{eq3.1} \Pr(\tau_{n}=\tau_{n+1}<T)=0, \quad n\in\N\cup\{0\}.
\end{equation}
Indeed, suppose that  $\omega \in \{\tau_{n}=\tau_{n+1}<T\}$. Then
there exists a sequence $\{t_m\}$ such that $t_m\searrow
\tau_n(\omega)$ and for every $m\in\N$,
\[
\frac{L_{\tau_n(\omega)}(\omega)
+U_{\tau_n(\omega)}(\omega)}{2}>U_{t_m}(\omega)
\quad\text{or}\quad\frac{L_{\tau_n(\omega)}(\omega)
+U_{\tau_n(\omega)}(\omega)}{2}< L_{t_m}(\omega).
\]
Since $L$ and $U$ are right-continuous, this implies that
\[
\frac{L_{\tau_n(\omega)}(\omega)+U_{\tau_n(\omega)}(\omega)}{2}
\geq U_{\tau_n(\omega)}(\omega)
\]
or
\[
\frac{L_{\tau_n(\omega)}(\omega)
+U_{\tau_n(\omega)}(\omega)}{2}\leq L_{\tau_n(\omega)}(\omega).
\]
Hence  $L_{\tau_n(\omega)}(\omega) =U_{\tau_n(\omega)}(\omega)$.
Since the barriers satisfy (B1), this shows (\ref{eq3.1}). We can
now prove that $\{\tau_n\}$ is of stationary type. Suppose, for
contradiction, that there is $\tau\in\T$ such that
$\tau_n\nearrow\tau$ and
$\Pr(\bigcap_{n=1}^\infty\{\tau_n<\tau\})>0$. Then
\[
\Pr\left(\frac{L_{\tau_{n-1}}+U_{\tau_{n-1}}}{2} \to
\frac{L_{\tau-}+U_{\tau-}}{2}\right)>0.
\]
This implies that $\Pr\big(L_{\tau-}
\geq\frac{L_{\tau-}+U_{\tau-}}{2}\geq U_{\tau-}\big)>0$, hence
that $P(L_{\tau-}\ge U_{\tau-})>0$, in contradiction with (B1).
Thus $\{\tau_n\}$ is of stationary type. Set
\[
H_t=\sum_{n=1}^\infty\frac{L_{\tau_{n-1}}+U_{\tau_{n-1}}}{2}
\ind{[\tau_{n-1},\tau_n)}(t),\quad t\in[0,T].
\]
Then $L_t\leq H_t\leq U_t$, $t\in[0,T]$, $\Pr$-a.e., and $H\in\V$,
because $\{\tau_n\}$ is of stationary type. Moreover, for each
$n\in\N$,
\[
\E|H|_{\tau_n}=\sum_{k=1}^{n}\E\left|\frac{U_{\tau_k}+L_{\tau_k}}{2}
- \frac{U_{\tau_{k-1}}+L_{\tau_{k-1}}}{2} \right|,
\]
which is finite because $L,U$ are of class D.
\end{proof}

The following example shows that in general there is no $H$
between barriers such that $\E{|H|_T}$ is finite.

\begin{example}\label{ex3.3}
Let $T=1$ and $\Fb=\{\F_t\}_{t\in[0,1]}$ be a Brownian filtration. Let $\{B_n\}_{n\in\N}$ be a partition of $\Omega$ such that $B_n$ is $\F_\frac{1}{4}$\,-measurable and $\Pr(B_n)=Cn^{-2}$ with $C=6\pi^{-2}$, $n\in\N$. Define ${h:[0,1)\to\R}$ by the formula
\[
h_t=\begin{cases} \frac{1}{2}, &
t\in[1-\frac{1}{2n+1},1-\frac{1}{2n+2}),\quad
n\in\N\cup\{0\},\smallskip\\
-\frac{3}{2}, & t\in[1-\frac{1}{2n},1-\frac{1}{2n+1}),\quad
n\in\N,
\end{cases}
\]
and put
\[
L_t=\sum_{n=1}^\infty h_{t\wedge(1-\frac{1}{n+1})}\ind{B_n},\quad
U_t=L_t+1, \quad t\in[0,T].
\]
One can check that $L,U$ satisfy the assumptions of Lemma
\ref{lm3.1}. Therefore there exists a process $H\in\V$ such that
$L_t\leq H_t\leq U_t$, $t\in[0,T]$, $\Pr$-a.s.
Consider now an arbitrary process $\bar H\in\V$ such that
$L_t\leq\bar H_t\leq U_t$, $t\in[0,T]$, $\Pr$-a.s. By the
construction of the barriers $L$ and $U$,
\[
|\bar
H|_T\ind{B_n}\geq\sum_{t\in[0,T]}(|U_t-L_{t-}|\wedge|U_{t-}-L_t|)
\ind{\{L_t-L_{t-}\neq 0\}}(t)\ind{B_n}=n\ind{B_n}.
\]
Hence
\[
\E{|\bar H|_T}=\sum_{n=1}^\infty \E{|\bar H|_T\ind{B_n}}\geq
\sum_{n=1}^\infty n\Pr(B_n)=\sum_{n=1}^\infty\frac{C}{n}=\infty.
\]
\end{example}
\medskip

Before proving our main result, we first introduce some additional
notation. Assume that $\xi,f$ satisfy (H1)--(H4), and $L,U$ are of
class D and satisfy (B1). Set
\[
\underline{f}_m(t,y)=f(t,y)-m(y-U_t)^+,\qquad
\overline{f}_n(t,y)=f(t,y)+n(L_t-y)^+.
\]
Then $\underline{f}_m,\overline{f}^n$ also satisfy (H1)--(H4),
because $y\mapsto n(L_t-y)^+$ and $y\mapsto m(y-U_t)^+$ are
Lipschitz continuous for $t\in[0,T]$ and $L,U$ are of class D. By
Theorem \ref{tw2.1}, for each $n\in\N$ there exists a unique
solution $(\overline{Y}^n,\overline{A}^n,\overline{M}^n)$ of
the equation $\urs(\xi,\overline{f}_n+dV,U)$, and for each $m\in N$ there
exists a unique solution
$(\underline{Y}^m,\underline{K}^m,\underline{M}^m)$ of
$\lrs(\xi,\underline{f}_m+dV,L)$. Therefore
\begin{align}
\label{eq3.16} \overline{Y}^n_t&=\xi+\int_t^T
f(s,\overline{Y}^n_s)\,ds +\int_t^TdV_s\nonumber \\
&\quad +\int_t^T n(L_s-\overline{Y}^n_s)^+\,ds - \int_t^T
d\overline{A}^n_s - \int_t^T d\overline{M}^n_s\leq U_t
\end{align}
and
\begin{align*}
\underline{Y}^m_t&=\xi+\int_t^T f(s,\underline{Y}^m_s)\,ds +
\int_t^TdV_s\\
&\quad -\int_t^T m(\underline{Y}^m_s-U_s)^+\,ds + \int_t^T
d\underline{K}^m_s - \int_t^T d\underline{M}^m_s\geq L_t
\end{align*}
for $t\in[0,T]$. The function $(t,y)\mapsto f(t,y)-m(y-U_t)^+
+n(L_s-y)^+$ also satisfies (H1)--(H4), so by \cite[Theorem
2.7]{klimsiak2014}, for any $n,m\in\N$  there exists a solution
$(Y^{n,m},M^{n,m})$ of the BSDE
\begin{align*}
{Y}^{n,m}_t&=\xi+\int_t^T f(s,Y^{n,m}_s)\,ds + \int_t^TdV_s
+\int_t^T n(L_s-{Y}^{n,m}_s)^+\,ds\\
&\quad -\int_t^T m({Y}^{n,m}_s-U_s)^+\,ds - \int_t^T d{M}^{n,m}_s,
\quad t\in[0,T].
\end{align*}
By Theorem \ref{tw2.1}, for each $m\in N$  the sequence
$\{Y^{n,m}\}_n$ is nondecreasing, for each $n\in\N$ the sequence
$\{Y^{n,m}\}_m$ is nonincreasing, and
\[
\underline{Y}^m_t=\sup_{n\in\N}{Y}^{n,m}_t =
\lim_{n\rightarrow\infty}{Y}^{n,m}_t,\quad
\overline{Y}^n_t=\inf_{m\in\N}{Y}^{n,m}_t
=\lim_{m\rightarrow\infty}{Y}^{n,m}_t,\quad t\in[0,T].
\]
In particular,  for all $n,m\in\N$ we have
\begin{equation}\label{eq3.2}
\overline{Y}^n_t\leq Y^{n,m}_t\leq\underline{Y}^m_t,\quad
t\in[0,T].
\end{equation}
By \cite[Proposition 2.1]{klimsiak2014}  the sequence
$\{\overline{Y}^n\}$ is nondecreasing, whereas the sequence
$\{\underline{Y}^m\}$ is nonincreasing. Set
\begin{equation}\label{eq3.03}
\underline{Y}_t=\inf_{m\in\N} \underline{Y}^m_t  =
\lim_{m\rightarrow\infty} \underline{Y}^m_t,\qquad
\overline{Y}_t=\sup_{n\in\N}\overline{Y}^n_t
=\lim_{n\rightarrow\infty}\overline{Y}^n_t.
\end{equation}
Since $\underline{Y}^m\ge L$ for all $m\in\N$ and
$\overline{Y}^n\le U$ for all $n\in\N$, we have
\[
\underline{Y}_t\geq L_t,\quad \overline{Y}_t\leq U_t,\quad
t\in[0,T],\quad P\mbox{-a.s.}
\]
Also note that from \eqref{eq3.2} and \eqref{eq3.03} and
monotonicity of the sequences $\{\overline{Y}^n\}$,
$\{\underline{Y}^m\}$ it follows that
\begin{equation}
\label{eq3.04} \overline{Y}^0\leq \overline{Y}^n\leq\overline{Y}
\leq\underline{Y}\leq\underline{Y}^m\leq\underline{Y}^0.
\end{equation}
Since $\overline{Y}^0$ and $\underline{Y}^0$ are solution of
reflected BSDEs, they are processes of class D.
\begin{lemma}\label{lm3.3}
Assume \emph{(H1), (H2)}. Then for every $r>0$
\[
t\mapsto \sup_{|y|\leq r} f(t,y) \in L^1(0,T).
\]
\end{lemma}
\begin{proof}
By (H1), for all $y\in[-r,r]$, $t\in[0,T]$ we have
\[
f(t,y)\geq f(t,r)-2\mu r,\qquad f(t,y)\leq f(t,-r)+2\mu r.
\]
Hence
\[
\sup_{|y|\leq r}|f(t,y)| \leq |f(t,-r)+2\mu r| \vee |f(t,r)-2\mu r|.
\]
It suffices to use (H2) to complete the proof.
\end{proof}

\begin{lemma}\label{lm3.4}
Let $(Y,K,A,M)$ be a solution of $\rs(\xi,f+dV,L,U)$, and let
$\tau\in\T.$ If $\xi\in L^1(\F_\tau),$ $f(t,y)\ind{(\tau,T]}(t)=0$ for all $y\in\R$ and
$t\in[0,T]$, and
\begin{equation}\label{eq3.01}
V_t=V_{t\wedge\tau},\quad L_t=L_{t\wedge\tau}, \quad
U_t=U_{t\wedge\tau},\quad t\in[0,T],
\end{equation}
then
\begin{equation}
\label{eq3.07} Y_t=Y_{t\wedge\tau},\quad K_t=K_{t\wedge\tau},\quad A_t=A_{t\wedge\tau},\quad
M_t=M_{t\wedge\tau},\quad t\in[0,T].
\end{equation}
\end{lemma}
\begin{proof}
By (LU4),
\begin{equation}\label{eq3.001}
Y_{t\wedge\tau}-Y_t = \int_{t\wedge\tau}^tdK_s
-\int_{t\wedge\tau}^tdA_s -\int_{t\wedge\tau}^tdM_s.
\end{equation}
Let $\{\zeta_n\}$ be a fundamental sequence for the local
martingale $M$, and let $\sigma\in\T$. Applying the Tanaka-Meyer
formula we get
\begin{align*}
\left(Y_{(\sigma\wedge\zeta_n)\wedge\tau}
-Y_{\sigma\wedge\zeta_n}\right)^+ &\leq
\int_{(\sigma\wedge\zeta_n)\wedge\tau}^{\sigma\wedge\zeta_n}
\ind{\{Y_{(s\wedge\tau)-}>Y_{s-}\}}\,dK_s\\
&\quad-\int_{(\sigma\wedge\zeta_n)\wedge\tau}^{\sigma\wedge\zeta_n}
\ind{\{Y_{(s\wedge\tau)-}>Y_{s-}\}}\, dA_s \\
&\quad-\int_{(\sigma\wedge\zeta_n)\wedge\tau}^{\sigma\wedge\zeta_n}
\ind{\{Y_{(s\wedge\tau)-}>Y_{s-}\}}\, dM_s\\
&\leq \int_{(\sigma\wedge\zeta_n)\wedge\tau}^{\sigma\wedge\zeta_n}
\ind{\{Y_{\tau}>Y_{s-}\}}\,dK_s -
\int_{(\sigma\wedge\zeta_n)\wedge\tau}^{\sigma\wedge\zeta_n}
\ind{\{Y_{\tau}>Y_{s-}\}}\,dM_s.
\end{align*}
Taking the expectation and then letting $n\to\infty$ yields
\[
\E(Y_{\sigma\wedge\tau}-Y_{\sigma})^+  \leq
E\int_{\sigma\wedge\tau}^{\sigma} \ind{\{Y_{\tau}>Y_{s-}\}}\,dK_s.
\]
On the other hand,  by (LU3) and \eqref{eq3.01},
\begin{align*}
\int_{\sigma\wedge\tau}^{\sigma} \ind{\{Y_{\tau}>Y_{s-}\}}\,dK_s &=\int_{\sigma\wedge\tau}^{\sigma} \ind{\{Y_{\tau}>Y_{s-}\}}
\ind{\{Y_{s-}=L_{s-}\}}\,dK_s \\
&= \int_{\sigma\wedge\tau}^{\sigma}
\ind{\{Y_{\tau}>Y_{s-}\}}\ind{\{Y_{s-}=L_{\tau}\}}\,dK_s.
\end{align*}
Hence
\begin{equation}\label{eq3.10}
E(Y_{\sigma\wedge\tau}-Y_{\sigma})^+ \leq E
\int_{\sigma\wedge\tau}^{\sigma}
\ind{\{Y_{\tau}>Y_{s-}\}}\ind{\{Y_{s-}=L_{\tau}\}}\,dK_s.
\end{equation}
From now on we consider the stopping time $\sigma$ defined by
\[
\sigma=\inf\{t>\tau:Y_{t\wedge\tau}>Y_{t} \}\wedge T.
\]
Observe that
\begin{equation}\label{eq3.12_5}
Y_{t\wedge\tau}\ind{\{t<\sigma\}}\leq Y_t\ind{\{t<\sigma\}},\quad
t\in[0,T].
\end{equation}
Set
\[
B_T = \Big\{\int_{\tau}^{T}
\ind{\{Y_{\tau}>Y_{s-}\}}\ind{\{Y_{s-}=L_{\tau}\}}\,dK_s>0\Big\}.
\]
Since $\ind{\{Y_{\tau}>Y_{s-}\}}\ind{\{Y_{s-}=L_{\tau}\}}
\leq\ind{\{L_\tau<Y_\tau\}}$, we have
\begin{equation}\label{eq3.11}
B_T\subset\{L_{\tau}<Y_{\tau}\}.
\end{equation}
From  (LU3), \eqref{eq3.12_5}, \eqref{eq3.11}  and the fact that
$L_t=L_{t\wedge\tau}$ for $t\in[0,T]$  it follows that
\begin{equation}\label{eq3.12_6}
\ind{B_T}\cdot\int_{\tau\wedge\sigma}^{\sigma} dK_s=0.
\end{equation}
By \eqref{eq3.10} and \eqref{eq3.12_6},
\begin{align*}
E(Y_{\tau}-Y_{\sigma})^+&\leq E\int_{\tau}^{\sigma}
\ind{\{Y_{\tau}>Y_{s-}\}}\ind{\{Y_{s-}=L_{\tau}\}}\,dK_s\nonumber\\
&=E\Big( \ind{B_T} \int_{\tau}^{\sigma}
\ind{\{Y_{\tau}>Y_{s-}\}}\ind{\{Y_{s-}=L_{\tau}\}}\,dK_s
\Big)=0.
\end{align*}
By the above,
\begin{equation}\label{eq3.12_4}
\E((Y_{\tau}-Y_{\sigma})^+\ind{\{\sigma=T\}})=0, \qquad\E(
(Y_{\tau}-Y_{\sigma})^+ \ind{\{\sigma<T\}})=0.
\end{equation}
Suppose that  $\Pr(\sigma=T)=1$. Then by  \eqref{eq3.12_5} and the
first equality in \eqref{eq3.12_4},\break $(Y_{t\wedge\tau}-Y_t)^+=0$
$\Pr$-a.s. for $t\in[0,T]$. We now prove that
\begin{equation}\label{eq3.12_7}
\Pr(\sigma<T)=0.
\end{equation}
By the second equality in \eqref{eq3.12_4},
\begin{equation}\label{eq3.12_8}
\Pr( \{Y_{\tau}\leq Y_{\sigma}\}\cap\{\sigma<T\}) =\Pr(\sigma<T).
\end{equation}
Observe that from the definition of $\sigma$ and the fact that
$L_t=L_{t\wedge\tau}$ for $t\in[0,T]$ it follows that
\begin{equation}\label{eq3.12_10}
\{\sigma<T\}\subset \{L_{\tau}<Y_{\tau}\}.
\end{equation}
Set
\[
\zeta=\inf\left\{t>\sigma:Y_t< \frac{Y_{\tau}+L_{\tau}}{2} \right\}.
\]
By right-continuity of $Y$ and \eqref{eq3.12_10} we have
$Y_{\zeta}\ind{\{\sigma<T\}} \leq
\frac{Y_{\tau}+L_{\tau}}{2}\ind{\{\sigma<T\}}$. By this  and
\eqref{eq3.12_10},
\begin{equation}\label{eq3.12_11}
\Pr\left( \{Y_{\zeta}<Y_{\tau}\}\cap\{\sigma<T\}\right) =\Pr\left(
\sigma<T\right).
\end{equation}
Furthermore, from \eqref{eq3.11},  the definition of $\zeta$ and
(LU3) it follows that
\begin{equation}\label{eq3.12_12}
0\leq\ind{B_T}\cdot\int_{\sigma}^\zeta dK_s
\leq\ind{\{L_\tau<Y_\tau\}}\cdot\int_{\sigma}^\zeta dK_s
=\ind{\{L_\tau<Y_\tau\}}\ind{\{\sigma<\zeta\}}\cdot\int_{\sigma}^\zeta dK_s=0.
\end{equation}
Observe that by the definition of the set $B_T$,
\[
\E\left( \int_{\tau}^{\zeta}
\ind{\{Y_{\tau}>Y_{s-}\}}\ind{\{Y_{s-}=L_{\tau}\}}\,dK_s \right)
=\E\left(\ind{B_T}
\int_{\tau}^{\zeta} \ind{\{Y_{\tau}>Y_{s-}\}}
\ind{\{Y_{s-}=L_{\tau}\}}\,dK_s\right).
\]
By the above equality, \eqref{eq3.12_6} and \eqref{eq3.12_12},
\begin{align*}
\E\left( \int_{\tau}^{\zeta}\ind{\{Y_{\tau}>Y_{s-}\}}\ind{\{Y_{s-}
=L_{\tau}\}}\,dK_s \right) &=\E\left(\ind{B_T} \int_{\tau}^{\sigma}
\ind{\{Y_{\tau}>Y_{s-}\}}
\ind{\{Y_{s-}=L_{\tau}\}}dK_s \right) \\
&\quad+ \E\left(\ind{B_T} \int_{\sigma}^{\zeta}
\ind{\{Y_{\tau}>Y_{s-}\}} \ind{\{Y_{s-}=L_{\tau}\}}\,dK_s \right)\\
&=0.
\end{align*}
This when combined with \eqref{eq3.10} with $\sigma$ replaced by
$\zeta$ gives $\E(Y_{\tau} - Y_{\zeta})^+=0$. Consequently, $E(
Y_{\tau}-Y_{\zeta})^+\ind{\{\sigma<T\}}=0$, which when combined
with \eqref{eq3.12_11} proves \eqref{eq3.12_7}. Thus
$(Y_{t\wedge\tau}-Y_{t})^+=0,$ $t\in[0,T]$, $\Pr$-a.s. Applying the
Tanaka-Meyer formula to the process $(Y_{t\wedge\tau}-Y_{t})^-$
and using similar arguments one can prove that
$(Y_{t\wedge\tau}-Y_{t})^-=0,$ $t\in[0,T]$, $\Pr$-a.s. Hence
\begin{equation}\label{eq3.052}
Y_t=Y_{t\wedge\tau},\quad t\in[0,T].
\end{equation}
From \eqref{eq3.001} and \eqref{eq3.052} we obtain
\[
0=\int_{t\wedge\tau}^tdK_s-\int_{t\wedge\tau}^tdA_s -
\int_{t\wedge\tau}^tdM_s,
\]
which implies (\ref{eq3.07}).
\end{proof}

\begin{theorem}\label{tw3.1}
Assume that {\em (H1) -- (H4), (B1), (B2)} are satisfied. Then
there exists a~unique solution $(Y,K,A,M)$ of $\rs(\xi,f+dV,L,U)$.
Moreover, $Y=\overline{Y}=\underline{Y}$.
\end{theorem}
\begin{proof}
By \cite[Corollary 3.2]{klimsiak2014}, there  exists at most one
solution of the equation $\rs(\xi,f+dV,L,U)$, so it suffices to prove the
existence of a solution. To this end, we first assume additionally
that $L,U$ are of class D. Then by Lemma \ref{lm3.1} there exists
$H\in\V$ such that $L_t\leq H_t\leq U_t$, $t\in[0,T]$ $\Pr$-a.s.
and $H_{\cdot\wedge\tau'_k}\in\Vc$ for some sequence $\{\tau'_k\}$
of stationary type. Set
\begin{equation}\label{eq3.06}
\tau_k=\tau'_k\wedge\delta_k
\end{equation}
and $H^{(k)}=H_{\cdot\wedge\tau_k}$, where
\[
\delta_k=\inf\left\{t\geq 0:\int_0^t f^{-}(s,H_s)\,ds>k
\right\}\wedge T.
\]
Observe that  $H^{(k)}\in\Vc$ and by Lemma \ref{lm3.3},
$\{\tau_k\}$  is of stationary type. The rest of the proof we
divide into 5 steps.

{\em Step 1.} We show the existence of a solution of
$\rs(\xi,f+dV,L,U)$ on stochastic intervals $[0,\tau_k]$. Set
\begin{align*}
U^{(k)}=U_{\cdot\wedge\tau_k},\qquad L^{(k)}
=L\ind{[0,\tau_k)}+(L_{\tau_k}\wedge\overline{Y}_{\tau_k})
\ind{[\tau_k,T]},\\
\xi^{(k)}=\overline{Y}_{\tau_k},\qquad f^{(k)}(\cdot,y)
=f(\cdot,y)\ind{[0,\tau_k]}, \qquad V^{(k)}=V_{\cdot\wedge\tau_k},
\end{align*}
where $\overline{Y}$ is defined by (\ref{eq3.03}).  By
\eqref{eq3.04}, $\xi^{(k)}\in L^1(\F_T)$. Also observe that
$L^{(k)}_T\leq \xi^{(k)}\leq U^{(k)}_T$ and $L^{(k)}_t\leq
H^{(k)}_t\leq U^{(k)}_t$, $t\in[0,T]$. Therefore by \cite[Theorem
3.3]{klimsiak2014}  there exists a unique solution
$(Y^{(k)},K^{(k)},A^{(k)},M^{(k)})$ of $\rs(\xi^{(k)},
f^{(k)}+dV^{(k)}, L^{(k)},U^{(k)})$ such that
\begin{equation}\label{eq3.02}
\E K^{(k)}_T<\infty,\qquad\E A^{(k)}_T<\infty.
\end{equation}
In particular, we have
\begin{align}\label{eq3.3}
Y^{(k)}_t&=\xi^{(k)}+\int_t^T f^{(k)}(s,Y^{(k)}_s)\,ds + \int_t^T dV^{(k)}_s\nonumber\\
&\quad + \int_t^TdK^{(k)}_s - \int_t^T dA^{(k)}_s -\int_t^T
dM^{(k)}_s
\end{align}
for $t\in[0,T]$. By Lemma \ref{lm3.4},
\begin{equation}\label{eq3.4}
(Y^{(k)}_t,K^{(k)}_t,A^{(k)}_t,M^{(k)}_t)
=(Y^{(k)}_{t\wedge\tau_k},K^{(k)}_{t\wedge\tau_k},A^{(k)}_{t\wedge\tau_k},
M^{(k)}_{t\wedge\tau_k}),\quad t\in[0,T].
\end{equation}

{\em Step 2.} We are going to show for every $\tau\in\T$,
\begin{equation}
\label{eq3.19} Y^{(k)}_\tau=\overline{Y}_{\tau\wedge\tau_k}.
\end{equation}
By Theorem \ref{tw2.1}, for each  $n\in\N$  there is a unique
solution $(Y^{(k),n},A^{(k),n},M^{(k),n})$ of
$\urs(\xi^{(k)},f^{(k),n}+dV^{(k)}, U^{(k)})$ with
$f^{(k),n}(t,y)=f^{(k)}(t,y)+n(L^{(k)}_t-y)^+$ and the triple
$(Y^{(k),n},A^{(k),n},M^{(k),n})$ satisfies the
equation
\begin{align}\label{eq3.054}
Y^{(k),n}_t&=\xi^{(k)}+\int_t^Tf^{(k)}(s,Y^{(k),n}_s)\,ds
+ \int_t^T dV^{(k)}_s \nonumber\\
&\quad + \int_t^T n(L^{(k)}_s-Y^{(k),n}_s)^+\,ds
-\int_t^TdA^{(k),n}_s - \int_t^T dM^{(k),n}_s,
\end{align}
and by \cite[Theorem 3.3]{klimsiak2014},
\begin{equation}\label{eq3.055}
Y^{(k),n}\nearrow Y^{(k)}.
\end{equation}
Write
$\tilde{Y}^n_t=Y^{(k),n}_{t}-\overline{Y}^n_{t}$,
$\tilde{A}^n_t=A^{(k),n}_{t}-\overline{A}^n_{t}$,
$\tilde{M}^n_t=M^{(k),n}_{t}-\overline{M}^n_{t}$. By
(\ref{eq3.16}), \eqref{eq3.054} and the Tanaka-Meyer formula, for
all $\zeta,\tau\in\T$ we have
\begin{align*}
\tilde{Y}^{n,+}_{\tau\wedge\zeta\wedge\tau_k}&
\leq\tilde{Y}^{n,+}_{\zeta\wedge\tau_k}
+\int_{\tau\wedge\zeta\wedge\tau_k}^{\zeta\wedge\tau_k}
\ind{\{\tilde{Y}^{n}_{s-}>0\}}(f^{(k)}(s,Y^{(k),n}_s)-f^{(k)}
(s,\overline{Y}^n_s))\,ds\nonumber\\
&\quad +\int_{\tau\wedge\zeta\wedge\tau_k}^{\zeta\wedge\tau_k}
 \ind{\{\tilde{Y}^{n}_{s-}>0\}}n\Big((L^{(k)}_s-Y^{(k),n}_s)^+
 -(L_s-\overline{Y}^{n}_s)^+\Big)\,ds\nonumber\\
&\quad -\int_{\tau\wedge\zeta\wedge\tau_k}^{\zeta\wedge\tau_k}
\ind{\{\tilde{Y}^{n}_{s-}>0\}}\,d\tilde{A}^n_s
-\int_{\tau\wedge\zeta\wedge\tau_k}^{\zeta\wedge\tau_k}
\ind{\{\tilde{Y}^{n,+}_{s-}>0\}}\,d\tilde{M}^n_s.
\end{align*}
By this and (H1),
\begin{align}
\label{eq3.6}
\tilde{Y}^{n,+}_{\tau\wedge\zeta\wedge\tau_k}&\leq
\tilde{Y}^{n,+}_{\zeta\wedge\tau_k}
+\mu\int_{\tau\wedge\zeta\wedge\tau_k}^{\zeta\wedge\tau_k}
\tilde{Y}^{n,+}_s\, ds \nonumber\\
&\quad +\int_{\tau\wedge\zeta\wedge\tau_k}^{\zeta\wedge\tau_k}
\ind{\{\tilde{Y}^{n}_{s-}>0\}}n\Big((L^{(k)}_s-Y^{(k),n}_s)^+
-(L_s-\overline{Y}^{n}_s)^+\Big)\,ds\nonumber\\
&\quad +\int_{\tau\wedge\zeta\wedge\tau_k}^{\zeta\wedge\tau_k}
\ind{\{\tilde{Y}^{n}_{s-}>0\}}\, d\overline{A}^n_{s}
-\int_{\tau\wedge\zeta\wedge\tau_k}^{\zeta\wedge\tau_k}
\ind{\{\tilde{Y}^{n}_{s-}>0\}}\,d\tilde{M}^n_s.
\end{align}
Since $y\mapsto (L_s-y)^+$ is  nonincreasing and
$L^{(k)}\ind{[0,\tau_k)}=L\ind{[0,\tau_k)}$, we have
\begin{equation}\label{eq3.7}
\int_{\tau\wedge\zeta\wedge\tau_k}^{\zeta\wedge\tau_k}
\ind{\{\tilde{Y}^{n}_{s-}>0\}}((L^{(k)}_s-Y^{(k),n}_s)^+
-(L_s-\overline{Y}^{n}_s)^+)\,ds\leq 0.
\end{equation}
Since $Y^{(k),n}_{t\wedge\tau_k}\le
U^{(k)}_{t\wedge\tau_k}=U^{(k)}_t$ and
$\overline{Y}^n_{t\wedge\tau_k}\le U_{t\wedge\tau_k}=U^{(k)}_t$,
we have 
\[
\overline{Y}^n_{t\wedge\tau_k}\leq
Y^{(k),n}_{t\wedge\tau_k}\vee\overline{Y}^n_{t\wedge\tau_k}\leq
U^{(k)}_t.
\]
Hence
\begin{align}\label{eq3.8}
\int_{\tau\wedge\zeta\wedge\tau_k}^{\zeta\wedge\tau_k}
\ind{\{\tilde{Y}^{n,+}_{s-}>0\}}\, d\overline{A}^n_{s}& \leq
\int_{\tau\wedge\zeta\wedge\tau_k}^{\zeta\wedge\tau_k}
\ind{\{\tilde{Y}^{n}_{s-}>0\}}\frac{Y^{(k),n}_{s-}
\vee\overline{Y}^n_{s-}-\overline{Y}^n_{s-}}{\tilde{Y}^n_{s-}}\,
d\overline{A}^n_{s}\nonumber\\
&\leq\liminf_{m\to\infty}m
\int_{\tau\wedge\zeta\wedge\tau_k}^{\zeta\wedge\tau_k}
\ind{\{\tilde{Y}^{n}_{s-}>\frac{1}{m}\}}(U_{s-}
-\overline{Y}^n_{s-})\,d\overline{A}^n_{s} =0.
\end{align}
By \eqref{eq3.6}--\eqref{eq3.8},
\[
\tilde{Y}^{n,+}_{\tau\wedge\zeta\wedge\tau_k}
\leq\tilde{Y}^{n,+}_{\zeta\wedge\tau_k}
+\mu\int_{\tau\wedge\zeta\wedge\tau_k}^{\zeta\wedge\tau_k}
\tilde{Y}^{n,+}_s\, ds
-\int_{\tau\wedge\zeta\wedge\tau_k}^{\zeta\wedge\tau_k}
\ind{\{\tilde{Y}^{n}_{s-}>0\}}\,d\tilde{M}^n_s
\]
for any $\tau,\zeta\in\T$. Let $\{\zeta_m\}$ be a fundamental
sequence for the local martingale $\tilde{M}^n$. Replacing $\zeta$
by $\zeta_m$ in the above inequality and then taking the
expectation we obtain
\[
\E\tilde{Y}^{n,+}_{\tau\wedge\zeta_m\wedge\tau_k}
\leq\E\tilde{Y}^{n,+}_{\zeta_m\wedge\tau_k}
+\mu\E\int_{\tau\wedge\zeta_m\wedge\tau_k}^{\zeta_m\wedge\tau_k}
\tilde{Y}^{n,+}_s \,ds.
\]
The processes $Y^{(k)},\overline{Y}^n$ are of class D as solutions
of reflected BSDEs. Consequently, $\tilde{Y}^{n,+}$ is of class D.
Therefore letting $m\to\infty$ in the above inequality we get
\begin{equation}\label{eq3.9}
\E\tilde{Y}^{n,+}_{\tau\wedge\tau_k}\leq\E\tilde{Y}^{n,+}_{\tau_k}
+\mu\E\int_{\tau\wedge\tau_k}^{\tau_k}\tilde{Y}^{n,+}_s\, ds
\end{equation}
for all $\tau\in\T$. Observe that
\begin{align*}
\int_{(\tau\vee t)\wedge\tau_k}^{\tau_k}\tilde{Y}^{n,+}_s\, ds
=\int_{((\tau\wedge\tau_k)\vee
t)\wedge\tau_k}^{\tau_k}\tilde{Y}^{n,+}_s\, ds &=\int_t^{T}
\tilde{Y}^{n,+}_s\ind{[\tau\wedge\tau_k,\tau_k]}(s)\,ds\\
&\leq
\int_t^T \tilde{Y}^{n,+}_{(\tau\vee s)\wedge\tau_k}\,ds.
\end{align*}
From the above inequality and \eqref{eq3.9} with $\tau$ replaced
by $\tau\vee t$ it follows that
\[
\E\tilde{Y}^{n,+}_{(\tau\vee
t)\wedge\tau_k} \leq\E\tilde{Y}^{n,+}_{\tau_k} +\mu\int_t^T
\E\tilde{Y}^{n,+}_{(\tau\vee s)\wedge\tau_k}\,ds,\quad
\tau\in\T,\; t\in[0,T].
\]
Applying Gronwall's inequality to the mapping
$t\mapsto\E\tilde{Y}^{n,+}_{(\tau\vee t)\wedge\tau_k}$ gives
\begin{equation}
\label{eq3.21} \E\tilde{Y}^{n,+}_{(\tau\vee t)\wedge\tau_k} \leq
e^{\mu T}\E\tilde{Y}^{n,+}_{\tau_k} \leq e^{\mu
T}\E|Y^{(k),n}_{\tau_k}-\overline{Y}^n_{\tau_k}|, \quad t\in[0,T].
\end{equation}
By (\ref{eq3.03}),  $\overline{Y}^n_{\tau_k}\nearrow
\overline{Y}_{\tau_k}=\xi^{(k)}$, whereas by (\ref{eq3.055}) and
\eqref{eq3.4}, $Y^{(k),n}_{\tau_k}\nearrow
Y^{(k)}_{\tau_k}=\xi^{(k)}$. Hence, by the monotone convergence
theorem,
\[
\E{|Y^{(k),n}_{\tau_k}-\xi^{(k)}|}\to 0,
\qquad\E{|\overline{Y}^n_{\tau_k}-\xi^{(k)}|}\to 0.
\]
Therefore applying Fatou's lemma and then \eqref{eq3.21} with
$t=T$ we obtain
\begin{align*}
\E\liminf_{n\to\infty}\tilde{Y}^{n,+}_{(\tau\wedge\tau_k)} &\leq
\liminf_{n\to\infty}\E\tilde{Y}^{n,+}_{\tau\wedge\tau_k} \\
&\leq
\liminf_{n\to\infty} e^{\mu T}(\E{|Y^{(k),n}_{\tau_k}-\xi^{(k)}|}
+ \E{|\overline{Y}^n_{\tau_k}-\xi^{(k)}|})=0.
\end{align*}
But $\tilde{Y}^{n}_{\tau\wedge\tau_k}\rightarrow
Y^{(k)}_{\tau\wedge\tau_k}-\overline{Y}_{\tau\wedge\tau_k}
=Y^{(k)}_{\tau}-\overline{Y}_{\tau\wedge\tau_k}$. Hence
$E(Y^{(k)}_{\tau}-\overline{Y}_{\tau\wedge\tau_k})^+=0$. In much
the same way one can show that
$E(Y^{(k)}_{\tau}-\overline{Y}_{\tau\wedge\tau_k})^-=0$, which
completes the proof of (\ref{eq3.19}). By (\ref{eq3.19}) and the
optional cross-section theorem \cite[p. 138-IV, (86)
Theorem]{DellMey1} the processes $Y^{(k)}$ and
$\overline{Y}_{\cdot\wedge\tau_k}$ are indistinguishable. In
particular, $\overline{Y}_{\cdot\wedge\tau_k}$ has c\`adl\`ag
trajectories. By the same method we show that $Y^{(k)}$ and
$\underline{Y}_{\cdot\wedge\tau_k}$ are indistinguishable.

{\it Step 3.} In this step we define a solution on $[0,T]$. By
Step~2, for every $k\in\N$,
\begin{equation}
\label{eq3.18}
Y^{(k)}_{t\wedge\tau_k}=\overline{Y}_{t\wedge\tau_k}
=\overline{Y}_{t\wedge\tau_k\wedge\tau_{k+1}}
=Y^{(k+1)}_{t\wedge\tau_k},\quad t\in[0,T].
\end{equation}
By (\ref{eq3.3}), (\ref{eq3.18}) and  uniqueness of the
semimartingale decomposition,
\[
(Y^{(k+1)}_{t\wedge\tau_k}, K^{(k+1)}_{t\wedge\tau_k},
A^{(k+1)}_{t\wedge\tau_k},
M^{(k+1)}_{t\wedge\tau_k})=(Y^{(k)}_{t\wedge\tau_k},
K^{(k)}_{t\wedge\tau_k}, A^{(k)}_{t\wedge\tau_k},
M^{(k)}_{t\wedge\tau_k}),\quad t\in[0,T].
\]
Therefore we may define processes $Y,K,A,M$ on $[0,T]$ by
\begin{equation}
\label{eq3.23} Y_t=Y^{(k)}_t,\quad K_t=K^{(k)}_t,\quad
A_t=A^{(k)}_t, \quad M_t=M^{(k)}_t,\quad t\in[0,\tau_k].
\end{equation}
By Step 2, $Y_{\tau\wedge\tau_k}=\underline{Y}_{\tau\wedge\tau_k}
=\overline{Y}_{\tau\wedge\tau_k}$ for all $\tau\in\T$ and
$k\in\N$, so letting $k\to\infty$ gives
$Y_{\tau}=\overline{Y}_{\tau}$ for
$\tau\in\T$. Hence, by the cross-section theorem,
\[
Y=\overline{Y}.
\]
The quadruple $(Y,K,A,M)$ is a solution of $\rs(\xi,f+dV,L,U)$.
Indeed, from (\ref{eq3.3}), (\ref{eq3.23}) and stationarity of
$\{\tau_k\}$ it follows that $(Y,K,A,M)$ satisfies (LU1) and
(LU4). Moreover, from the fact that
$(Y^{(k)},K^{(k)},A^{(k)},M^{(k)})$ is a solution of
$\rs(\xi^{(k)}, f^{(k)}+dV^{(k)}, L^{(k)},U^{(k)})$ and
(\ref{eq3.23}) it follows that  $L_{t\wedge\tau_k}\leq
Y_{t\wedge\tau_k}\leq U_{t\wedge\tau_k}$, $t\in[0,T]$,
$\Pr$-a.s. and
\[
\int_0^{\tau_k}(Y_{t-}-L_{t-})\,dK_t =\int_0^{\tau_k}
(U_{t-}-Y_{t-})\,dA_t=0
\]
for $k\in\N$. Since  $\{\tau_k\}$ is of stationary type, this
implies (LU2) and~(LU3).

{\it Step 4.} Repeating the arguments from step 3 and step 4 for $\xi^{(k)}=\underline{Y}_{\tau_k}$ we prove that $\underline{Y}=Y$, where $(Y,K,A,M)$ is a solution of $\rs(\xi,f+dV,L,U)$. Therefore, by the uniqueness of solution, $Y=\overline{Y}=\underline{Y}.$

{\em Step 5.} We now show how to dispense with the assumption that
$L,U$ are of class D.

Let $\wideutilde{Y}, \widetilde{Y}$ be processes appearing in
Theorem \ref{tw2.1}. 
By \cite[Proposition 2.1]{kliroz},
\[
\widetilde{Y}_t\leq\wideutilde{Y}_t,\quad t\in[0,T],\quad
P\mbox{-a.s.}
\]
Let $\varepsilon>0$, and let $L^\e_t=L_t\vee(\widetilde{Y}_t-\e)$,
$U^\e_t=U_t\wedge(\wideutilde{Y}_t+\e)$. If $L,U$ satisfy (B1),
(B2), then also $L^\e,U^\e$ satisfy (B1) and are processes of
class (D). By steps 1-3 there exists a unique solution $(Y,K,A,M)$
of $\rs(\xi,f+dV,L^\e,U^\e)$. As in the proof of Theorem \ref{tw2.1}, one can check that $(Y,K,A,M)$ is also solution of
$\rs(\xi,f+dV,L,U)$.
\end{proof}

\begin{corollary}\label{wn3.6}
Assume that $L,U$ satisfy \mbox{\rm(B1), (B2)}. Then there exists\break
a~semimartingale $Y$ of class \mbox{\rm D} such that $L_t\leq
Y_t\leq U_t$, $t\in[0,T]$ $\Pr$-a.s.
\end{corollary}
\begin{proof}
It is enough to consider $\xi=L_T^+\wedge U_T$, $f\equiv 0$,
$V\equiv 0$, and apply Theorem \ref{tw3.1}.
\end{proof}

\begin{remark}\label{rem3.6}
Let $\{\tau_n\}$ be a sequence defined by \eqref{eq3.06}. If there
exists  $k_0\in\N$ such that
\begin{equation}\label{eq3.27}
\Pr(\tau_{k_0}=T)=1,
\end{equation}
then by from Step 3 of the proof of Theorem \ref{tw3.1} it follows that  
\[
(Y,K,A,M)=(Y^{(k_0)}, K^{(k_0)}, A^{(k_0)}, M^{(k_0)})
\] 
is a solution of $\rs(\xi,f+dV, L,U)$. Furthermore, by \eqref{eq3.02},
$\E K_T<\infty$ and $\E A_T<\infty$, and by \cite[Lemma
2.3]{kliroz}, $f(\cdot,Y)\in L^1(\F)$. Also note that a sufficient
condition for (\ref{eq3.27}) to hold is the following: there is
$H\in\V^1$ such that $L_t\le H_t\le U_t$, $t\in [0,T]$, and
$t\rightarrow f(t,H_t)$ is bounded from below.
\end{remark}

The following example shows that in general $\E{K_T}$  and
$\E{A_T}$ need not  be finite even if $f\equiv 0$ and $V\equiv0$.

\begin{example}
Let  $\mathbb{F}$ be a Brownian filtration, and let $L,U$ be
defined as in Example \ref{ex3.3}. Set $\xi=(L_T+U_T)/2$ and
$f\equiv0$, $V\equiv0$. By Theorem \ref{tw3.1}, there exists a
unique solution $(Y,K,A,M)$ of $\rs(\xi,0,L,U)$. In particular,
\[
Y_t=\xi+\int_t^T dK_s - \int_t^T dA_s - \int_t^T dM_s,\quad
t\in[0,T].
\]
Let $\tau_n=1-\frac{1}{n}$. Since the filtration is Brownian,
$\Delta M_{\tau_n}=0$ $\Pr$-a.s. for every $n\in\N$. Hence
\[
\Delta Y_{\tau_n}= \Delta A_{\tau_n}-\Delta K_{\tau_n},\quad
n\in\N.
\]
In fact,  by (LU2), (LU3) and the definitions of $L$ and $U$,
$\Delta Y_{\tau_m}=\Delta A_{\tau_m}$ if $m$ is even and $\Delta
Y_{\tau_m}=-\Delta K_{\tau_m}$ if $m$ is odd. From this and the
fact that $L\le Y\le U$ it follows that
\[
\Pr(\{\Delta A_{\tau_m}\geq 1\} \cap B_n) =Cn^{-2},\quad 2\leq
m\leq n+1
\]
when $m$ is even, and
\[
\Pr(\{\Delta K_{\tau_m}\geq 1\} \cap B_n) = Cn^{-2},\quad 2\leq
m\leq n+1
\]
when $m$ is odd. Hence
\[
\E{K_T}=\E{|K|_T}=\sum_{n=1}^\infty \E{|K|_T\ind{B_n}}\geq
\sum_{n=2}^\infty
\frac{n-1}{2}\Pr(B_n)=C\sum_{n=2}^\infty\frac{n-1}{2n^2}=\infty
\]
and
\[
\E{A_T}=\E{|A|_T}=\sum_{n=1}^\infty \E{|A|_T\ind{B_n}}\geq
\sum_{n=2}^\infty
\frac{n-1}{2}\Pr(B_n)=C\sum_{n=2}^\infty\frac{n-1}{2n^2}=\infty.
\]
\end{example}

\section{Dynkin games}
\label{sec4}

In this section we consider a certain stochastic game of stopping
called Dynkin game. For an interpretation of notions which we
define below (payoff function, lower and upper value of the game)
we defer the reader to \cite{CvitKar1996}.

Let $L,U$ be \cadlag processes of class D such that $L_t\leq U_t$,
$t\in[0,T]$, $\Pr$-a.s., and let $f,\xi,V$ be as in Section
\ref{sec3}. Also assume that conditions (H1)--(H4) are satisfied.
Consider a stopping game with payoff function
\begin{align}\label{eq.payoff}
R_t(\sigma,\tau)&=\int_t^{\sigma\wedge\tau}f(s,Y_s)\,ds +
\int_t^{\sigma\wedge\tau} dV_s \nonumber\\
&\quad+ \xi\ind{\{\sigma\wedge\tau=T\}}
+L_\tau\ind{\{\tau<T,\tau\leq\sigma\}}
+U_\sigma\ind{\{\sigma<\tau\}},\quad \sigma,\tau\in\T_t,
\end{align}
where $(Y,K,A,M)$ is a solution of $\rs(\xi,f+dV, L,U)$ such that
$K,A\in\V_0^1.$ By Remark \ref{rem3.6} such solution exists if (B1),
(B2) and \eqref{eq3.27} are satisfied.

The lower value $\underline{V}$ and the upper value $\overline{V}$
of the stochastic game corresponding to $R$ are defined by
\[
\underline{V}_t=\esssup_{\tau\in\T_t}\essinf_{\sigma\in\T_t}\E(R_t(\sigma,\tau)|
\F_t),\qquad
\overline{V}_t=\essinf_{\sigma\in\T_t}\esssup_{\tau\in\T_t}\E(R_t(\sigma,\tau)|
\F_t).
\]
We say that the game has a value if
$\underline{V}_t=\overline{V}_t$, $t\in[0,T]$, $\Pr$-a.s.

\begin{lemma}\label{lm4.1}
Let $\{\tau_n\}$ be a sequence of stopping times such that
$\tau_n\nearrow \tau$ $\Pr$-a.s., and
\begin{equation}\label{eq4.1}
\Pr(\liminf_{n\to\infty}\{\tau_n=\tau\})=1.
\end{equation}
Then for every $\sigma\in\T_t$, $ \E(R_t(\sigma,\tau_n)|\F_t)\to
\E(R_t(\sigma,\tau)|\F_t)$ $\Pr$-a.s. as $n\rightarrow\infty$.
\end{lemma}
\begin{proof}
By \eqref{eq.payoff} and \eqref{eq4.1}, $R_t(\sigma,\tau_n)\to
R_t(\sigma,\tau)$ $\Pr$-a.s. Since  $V,L,U$ are of class D and
$\E|\xi|+\E\int_0^T|f(t,Y_t)|dt<\infty$, we conclude  from
(\ref{eq.payoff}) that the family $\{R_t(\sigma,\tau_n)\}_{n\in\N}$ is a
uniformly integrable family of random variables. Therefore the
desired convergence follows from \cite[Theorem 1.3]{LipcerSziriajewEN}.
\end{proof}

\begin{theorem}
\label{tw4.2} Let the assumptions of Theorem \ref{tw3.1} hold, and
additionally \mbox{\rm(\ref{eq3.27})} is satisfied. Then the
stochastic game corresponding to the payoff function
\eqref{eq.payoff} has the value equal to the first component of
the solution of $\rs(\xi,f+dV,L,U)$, i.e.
\begin{equation}
\label{eq4.2} Y_t=\underline{V}_t=\overline{V}_t,\quad
t\in[0,T],\quad P\mbox{-}a.s.
\end{equation}
\end{theorem}
\begin{proof}
By \cite[Lemma 5.3]{Lepeltier2007}, to show that the game has a
value it suffices to prove that  for any $\e>0$ and $t\in[0,T]$
there exist $\sigma_t^\e,\tau_t^\e\in\T_t$ such that for all
$\sigma,\tau\in\T_t$,
\begin{equation}
\label{eq4.3} -\e+\E(R_t(\sigma_t^\e,\tau)| \F_t) \leq
\E(R_t(\sigma,\tau_t^\e)| \F_t)+\e.
\end{equation}
To show (\ref{eq4.3}), we set  $\sigma_t^\e=\inf\{s>t: Y_s\geq
U_s-\e\}\wedge T$. Observe that $Y_{s-}<U_{s-}$ for
$t<s\le\sigma_t^\e$, and hence, by (LU3),
\begin{equation}\label{eq4.4}
A_s\ind{(t,\sigma^\e_t]}(s)=A_{\sigma_t^\e}\ind{(t,\sigma^\e_t]}(s),\quad s\in[0,T].
\end{equation}
Clearly, for any $\tau\in\T_t$,
\[
\{\sigma_t^\e=T\}\subset \{\tau\leq \sigma_t^\e\},\qquad \{\tau >
\sigma_t^\e\} \subset \{\sigma_t^\e < T\}.
\]
By this and  \eqref{eq4.4} it follows that on the set $\{\tau\leq \sigma_t^\e\}$ we have
\begin{align*}
R_t(\sigma_t^\e,\tau)&=\int_t^{\tau}f(s,Y_s)\,ds
+ \int_t^{\tau} dV_s + \xi\ind{\{\tau=T\}}+L_\tau\ind{\{\tau<T\}}\\
&\leq \int_t^{\tau}f(s,Y_s)\,ds + \int_t^{\tau} dV_s
+\xi\ind{\{\tau=T\}}+Y_\tau\ind{\{\tau<T\}} + \int_t^{\tau}dK_s
-\int_t^{\tau}dA_s\\
&\leq Y_t +\int_t^\tau dM_s,
\end{align*}
whereas on $\{\tau>\sigma_t^\e\}$ we have
\begin{align*}
R_t(\sigma_t^\e,\tau)&=\int_t^{\sigma_t^\e}f(s,Y_s)\,ds
+ \int_t^{\sigma_t^\e} dV_s +U_{\sigma_t^\e}\ind{\{\sigma_t^\e<\tau\}}\\
&\leq \int_t^{\sigma_t^\e}f(s,Y_s)\,ds
+ \int_t^{\sigma_t^\e} dV_s + Y_{\sigma_t^\e}
+ \int_t^{\sigma_t^\e}dK_s - \int_t^{\sigma_t^\e}dA_s +\e\\
&= Y_t + \int_t^{\sigma_t^\e} dM_s +\e.
\end{align*}
Hence
\[
R_t(\sigma_t^\e,\tau)=R_t(\sigma_t^\e,\tau)\ind{\{\tau\leq
\sigma_t^\e\}} + R_t(\sigma_t^\e,\tau)\ind{\{\tau > \sigma_t^\e\}}
\leq Y_t + \int_t^{\sigma_t^\e\wedge\tau} dM_s +\e
\]
$\Pr$-a.s. Let $\{\zeta_n\}$ be a fundamental sequence for the
local martingale $M$, and let $\tau_n= \tau\wedge\zeta_n$. Then
$\{\tau_n\}$ satisfies the assumptions of Lemma \ref{lm4.1} and
\[
E(R_t(\sigma_t^\e,\tau\wedge\zeta_n)|\F_t)\leq E\left(Y_t +
\int_t^{\sigma_t^\e\wedge\tau\wedge\zeta_n} dM_s
+\e\Bigg|\F_t\right)=Y_t+\e.
\]
Letting $n\rightarrow\infty$ and using Lemma \ref{lm4.1} we obtain
\begin{equation}\label{eq.payoff1}
\E(R_t(\sigma_t^\e,\tau)|\F_t)\leq Y_t +\e.
\end{equation}
Now, let us consider the stopping time
$\tau_t^\e=\inf\{s>t:Y_s\leq L_s+\e\}\wedge T$. Analysis similar
to that in the proof of (\ref{eq.payoff1}) shows that for any
$\e>0$ and $\sigma\in\T_t$,
\begin{equation}\label{eq.payoff2}
\E(R_t(\sigma,\tau_t^\e)|\F_t)\geq Y_t -\e.
\end{equation}
Combining \eqref{eq.payoff1} with \eqref{eq.payoff2} we see that
for any $\varepsilon>0$,
\begin{equation}
\label{eq4.5} -\e+\E(R_t(\sigma_t^\e,\tau)| \F_t)\le Y_t\le
\E(R_t(\sigma,\tau_t^\e)| \F_t)+\e.
\end{equation}
Thus (\ref{eq4.3}) is satisfied, and, in consequence, the game has
a value.   Moreover, from (\ref{eq4.5}) and the definitions of
$\underline{V},\overline{V}$ it follows that
$-\e+\overline{V}_t\leq Y_t\leq\underline{V}_t+\e$, $t\in[0,T]$,
for $\e>0$. Since we already know that the game has a value, this
implies (\ref{eq4.2}).
\end{proof}

Note that Dynkin games were studied, in different contexts, by
several authors. For results related to Theorem \ref{tw4.2} we
defer the reader to
\cite{CvitKar1996,klimsiak2014,Lepeltier2007,Stet1982,Zab1984} and
references given there.

\end{document}